\font \eightrm=cmr8

\font\cyr=wncyr10

\documentclass[10pt]{amsart}
\usepackage{amssymb}

\usepackage[all]{xy}
\usepackage{epsfig,exscale}
\usepackage{xspace}
\usepackage{colordvi}

\usepackage{amsmath}
\usepackage{amsthm}
\usepackage{amscd}
\usepackage{eucal}
\usepackage{color}
\usepackage[small,nohug,heads=littlevee]{diagrams}
\newcommand{\nc}{\newcommand}

\newtheorem{thm}{Theorem}
\newtheorem{exam}{Example}

\newtheorem{cor}[thm]{Corollary}

\newtheorem{lem}[thm]{Lemma}
\newtheorem{prop}[thm]{Proposition}

\newtheorem{rmk}[thm]{Remark}

\def\sha{{\mbox{\cyr X}}}

\nc{\ignore}[1]{{}}
\nc{\mrm}[1]{{\rm #1}}
\nc{\dirlim}{\displaystyle{\lim_{\longrightarrow}}\,}
\nc{\invlim}{\displaystyle{\lim_{\longleftarrow}}\,}
\nc{\vep}{\varepsilon} \nc{\ep}{\epsilon}
\nc{\sigmat}{\widetilde\sigma}
\nc{\ostar}{\overline{*}}

\nc{\mchar}{\mrm{Char}} \nc{\Hom}{\mrm{Hom}} \nc{\id}{\mrm{id}}

\nc{\remark}{\noindent{\bf{Remark:}}}
\nc{\remarks}{\noindent{\bf{Remarks:}}}

 \nc{\delete}[1]{}
 \nc{\grad}[1]{^{({#1})}}
 \nc{\fil}[1]{_{#1}}

\nc{\BA}{{\Bbb A}} \nc{\CC}{{\Bbb C}} \nc{\DD}{{\Bbb D}}
\nc{\EE}{{\Bbb E}} \nc{\FF}{{\Bbb F}} \nc{\GG}{{\Bbb G}}
\nc{\HH}{{\Bbb H}} \nc{\LL}{{\Bbb L}} \nc{\NN}{{\Bbb N}}
\nc{\PP}{{\Bbb P}} \nc{\QQ}{{\Bbb Q}} \nc{\RR}{{\Bbb R}}
\nc{\TT}{{\Bbb T}} \nc{\VV}{{\Bbb V}} \nc{\ZZ}{{\Bbb Z}}
\nc{\Cal}[1]{{\mathcal {#1}}}
\nc{\mop}[1]{\mathop{\hbox {\rm #1} }}
\nc{\smop}[1]{\mathop{\hbox {\eightrm #1} }}
\nc{\mopl}[1]{\mathop{\hbox {\rm #1} }\limits}
\nc{\frakg}{{\frak g}}
\nc{\g}[1]{{\frak {#1}}}
\def \restr#1{\mathstrut_{\textstyle |}\raise-8pt\hbox{$\scriptstyle #1$}}
\def \srestr#1{\mathstrut_{\scriptstyle |}\hbox to
  -1.5pt{}\raise-4pt\hbox{$\scriptscriptstyle #1$}}
\nc{\wt}{\widetilde} \nc{\wh}{\widehat}

\nc{\redtext}[1]{\textcolor{red}{#1}}
\nc{\bluetext}[1]{\textcolor{blue}{#1}}

\nc\fleche[1]{\mathop{\hbox to #1 mm{\rightarrowfill}}\limits}
\def\semi{\mathrel{\times}\kern -.85pt\joinrel\mathrel{\raise 1.4pt\hbox{${\scriptscriptstyle |}$}}}

\begin{document}

\title[New identities in dendriform algebras]
      {New identities in dendriform algebras}

\author{Kurusch Ebrahimi-Fard}
\address{Max-Planck-Institut f\"ur Mathematik,
         Vivatsgasse 7,
         D-53111 Bonn, Germany.}
 \email{kurusch@mpim-bonn.mpg.de}
 \urladdr{http://www.th.physik.uni-bonn.de/th/People/fard/}

\author{Dominique Manchon}
\address{Universit\'e Blaise Pascal,
         C.N.R.S.-UMR 6620,
         63177 Aubi\`ere, France}
         \email{manchon@math.univ-bpclermont.fr}
         \urladdr{http://math.univ-bpclermont.fr/~manchon/}

\author{Fr\'ed\'eric Patras}
\address{Laboratoire J.-A. Dieudonn\'e
         UMR 6621, CNRS,
         Parc Valrose,
         06108 Nice Cedex 02, France}
\email{patras@math.unice.fr} \urladdr{www-math.unice.fr/~patras}

\date{July 1, 2007\\
\noindent {\footnotesize{${}\phantom{a}$ 2001 PACS Classification:
03.70.+k, 11.10.Gh, 02.10.Hh}} }

\begin{abstract}
Dendriform structures arise naturally in algebraic combinatorics
(where they allow, for example, the splitting of the shuffle
product into two pieces) and through Rota--Baxter algebra
structures (the latter appear, among others, in differential
systems and in the  renormalization process of pQFT). We prove new
combinatorial identities in dendriform dialgebras that appear to
be strongly related to classical phenomena, such as the
combinatorics of Lyndon words, rewriting rules in Lie algebras, or
the fine structure of the Malvenuto--Reutenauer algebra. One of
these identities is an abstract noncommutative, dendriform,
generalization of the Bohnenblust--Spitzer identity and of an
identity involving iterated Chen integrals due to C.S.~Lam.
\end{abstract}

\maketitle



\section{Introduction}
\label{sect:introd}

Many associative algebras arising from combinatorial constructions
have a nice property: their product can be splitted into two
components that behave nicely with respect to the original
product. The best known example of this phenomenon is given by
Ree's recursive definition and study of the shuffle product
\cite{ree57}. The identities underlying this splitting, now called
dendriform identities, can be traced back to the work of
M.P.~Sch\"utzenberger on the properties of Hall basis of free Lie
algebras \cite{sch58}\footnote{The third author thanks
Ch.~Reutenauer for communication of \cite{sch58}.}. However, in
spite of Sch\"utzenberger's seminal ideas, and of the regular use
of the shuffle product --splitting--, e.g. in combinatorics or
algebraic topology, dendriform structures were not investigated
for their own till recently.

The situation has changed and, as explained below, dendriform
algebras have risen a considerable interest. The purpose of the
present article is to derive new dendriform identities and to
study their applications to classical problems and structures in
algebraic combinatorics, such as rewriting rules in free Lie
algebras, properties of Lyndon words in relation with free Lie
algebra basis, or Dynkin-type identities in the
Malvenuto--Reutenauer algebra.

In abstract terms, a dendriform algebra is an algebra with left
and right commuting representations on itself, written $\prec$ and
$\succ$, such that $x(y) = x \succ y$ and $(y)x = y \prec x$. The
two actions add to form the product of the algebra. In the case of
a commutative algebra, as an extra axiom, the left and right
actions are further required to identify canonically (so that $x
\succ y = y \prec x$, the particular case investigated in depth by
Sch\"utzenberger in~\cite{sch58}).

J.-L. Loday recently formalized this structure by introducing
so-called dendriform identities in connection with dialgebra
structures. Free dendriform algebras were described in terms of
trees in \cite{Lod01} (in fact, free commutative dendriform
algebras had been described in \cite{sch58}). Following the work
of M.~Aguiar~\cite{Ag00}, the first author of the present article
constructed then a forgetful functor from associative Rota--Baxter
algebras to dendriform algebras, as well as various forgetful
functors from dendriform algebras to other types of algebras
\cite{E}. Since Rota--Baxter algebras arise in many mathematical
contexts (such as fluctuation theory, integral and finite
differences calculus or perturbative renormalization in quantum
field theory), this construction provides the theory of dendriform
structures with a whole variety of new examples, besides the
classical ones arising from shuffle algebras (such as the
classical shuffle algebra or the algebras of singular cochains in
algebraic topology). This discovery was one of the leading
motivation of the present article, that extends to the dendriform
context ideas that have been developed by the authors, partly with
J.M.~Gracia-Bond\'{\i}a, in the setting of Rota--Baxter
algebras~\cite{EGP07,EMP07}, and that generalize to the
noncommutative Rota--Baxter and dendriform setting classical
results such as the Bohnenblust--Spitzer formula of fluctuation
theory~\cite{Rota} or Lam's identities for iterated integrals and
solutions of first order linear differential equations~\cite{Lam}.

Other results should be quoted here that have contributed to the
development of the theory of dendriform structures. F.~Chapoton~\cite{chapoton2000} (resp. M.~Ronco
\cite{ronco2002})
discovered that the classical proof of the Cartier--Milnor--Moore
theorem~\cite{mm65} (respectively its modern combinatorial proof
\cite{pat94}) could be extended to bialgebras with a dendriform
structure, linking dendriform structures with other algebraic
structures such as brace and pre-Lie algebras. Aguiar established
in~\cite{Ag02} unexpected connections with the infinitesimal
bialgebra structures studied in \cite{Ag99,Ag01}. Another striking
result in the field, and a great recent achievement in algebraic
combinatorics, is due to L.~Foissy, who was able to prove the
Duchamp--Hivert--Thibon conjecture (the Lie algebra of primitive
elements of the Malvenuto--Reutenauer Hopf algebra is a free Lie
algebra) using another dendriform version of the
Cartier--Milnor--Moore theorem \cite{foi05}. Other applications to
algebraic combinatorics have been developed recently by F.~Hivert,
J.~Novelli and Y.~Thibon \cite{NT06,HNT07}. These various results,
together with the classical identities in free Lie algebras
arising from the combinatorics of shuffles and of Hall and Lyndon
basis, contributed strongly to motivate further the present
article, and to the applications considered below of identities in
dendriform algebras to questions in algebraic combinatorics.


\section{Operations on dendriform algebras}
\label{formulation}

In concrete terms, a {\sl dendriform algebra\/} (or dendriform dialgebra)~\cite{Lod01}
over a field $k$ is a $k$-vector space $A$ endowed with two
bilinear operations $\prec$ and $\succ$ subject to the three
axioms below:
\begin{eqnarray}
 (a \prec b) \prec c  &=& a \prec (b*c)        \label{A1}\\
 (a \succ b) \prec c  &=& a \succ (b\prec c)   \label{A2}\\
  a \succ (b \succ c) &=& (a*b) \succ c        \label{A3},
\end{eqnarray}
where $a*b$ stands for $a \prec b + a \succ b$. These axioms
easily yield associativity for the law $*$. See~\cite{sch58} for
the commutative version, i.e. when furthermore $a \prec b = b \succ a$.

\begin{exam} The shuffle dendriform algebra. {\rm{The tensor algebra
$T(X)$ over an ordered alphabet is the linear span of the words
(or noncommutative monomials) $y_1 \ldots y_n$, $y_i \in X$ (we
will also use, when convenient, the notation $(y_1, \ldots ,y_n)$
for $y_1 \ldots y_n$). The concatenation product on $T(X)$ is
written by a dot: $y_1 \ldots y_n\ \cdot \ z_1 \ldots z_k:= y_1
\ldots y_n z_1 \ldots z_k$. The tensor algebra is provided recursively with a
dendriform algebra structure by the identities:
$$
    y_1 \ldots y_n \prec z_1 \ldots z_k := y_1 (y_2 \ldots y_n \prec z_1 \ldots z_k + y_2 \ldots y_n \succ z_1 \ldots z_k)
$$
$$
    y_1 \ldots y_n \succ  z_1 \ldots z_k := z_1(y_1 \ldots y_n \prec z_2 \ldots z_k + y_1 \ldots y_n \succ z_2 \ldots z_k)
$$
Of course, this is nothing but a rewriting of Ree's recursive
definition of the shuffle product $\sha$, to which the associative
product $\prec + \succ$ identifies \cite{ree57,sch58}.}}
\end{exam}

\begin{exam} \label{ex:MAX} The MAX dendriform algebra. {\rm{For any word
$w$ over the ordered alphabet $X$, let us write $max(w)$ for the
highest letter in $w$. The tensor algebra is provided with another
dendriform algebra structure by the identities:
$$
    u\succ v = u \cdot v\ \ \mbox{\rm if}\ \ max(u)<max(v) \ \ \mbox{\rm and}\ \ 0 \ \ \mbox{else}
$$
$$
    u\prec v = u \cdot v\ \ \mbox{\rm if}\ \ max(u) \geq max(v) \ \ \mbox{\rm and}\ \ 0 \ \ \mbox{else}
$$
where $u$ and $v$ run over the words over $X$. The associative
product $\prec + \succ$ identifies with the concatenation product.
MAX dendriform structures have appeared in the setting of
noncommutative generalizations of the algebra of symmetric
functions \cite{NT06,HNT07}.}}
\end{exam}

\begin{exam} The Malvenuto--Reutenauer dendriform
algebra. {\rm{Let us write ${\bf S}_\ast$ for the
Malvenuto--Reutenauer algebra, that is, the direct sum of the
group algebras of the symmetric groups ${\bf Q}[S_n]$, equipped
with the (shifted) shuffle product (written $\ast$):
$$
    \forall (\sigma , \beta )\in S_n \times S_m, \ \sigma\ast \beta :=(\sigma(1), \ldots ,\sigma(n))
                                                   \sha (\beta (1) + n, \ldots ,\beta (m)+n)
$$
The restriction to ${\bf S}_\ast$ of the dendriform structure on
the tensor algebra provides ${\bf S}_\ast$ with a dendriform
algebra structure:
$$
    \sigma \prec \beta := \sigma (1)\cdot ((\sigma(2), \ldots ,\sigma(n))\sha (\beta(1)+n, \ldots ,\beta(m)+n))
$$
$$
    \sigma \succ \beta :=(\beta (1)+n)\cdot ((\sigma(1), \ldots ,\sigma(n))\sha (\beta(2)+n, \ldots ,\beta(m)+n)).
$$
This structure is essentially the one used by Foissy to prove the
Duchamp--Hivert--Thibon conjecture \cite{foi05}.}}
\end{exam}

\begin{exam} Dendriform algebras of linear operators. {\rm{Let
$\mathcal A$ be any algebra of operator-valued functions on the
real line, closed under integrals $\int_0^x$. One may wish to
consider, for example, smooth $n \times n$ matrix-valued
functions. Then, $\mathcal A$ is a dendriform algebra for the
operations:
$$
    A \prec B (x):=A(x)\cdot \int\limits_0^xB(t)dt
 \qquad
    A \succ B (x):=\int\limits_0^xA(t)dt\cdot B(x)
$$
with $A,B \in \mathcal A$. This is a particular example of a
dendriform structure arising from a Rota--Baxter algebra
structure. We refer to the last section of the article for further
details on Rota--Baxter algebras and their connections to
dendriform algebras. Here, let us simply mention that the
Rota--Baxter operator on $\mathcal A$ giving rise to the
dendriform structure is: $R(A)(x):=\int\limits_0^xA(t)dt$.
Aguiar~\cite{Ag00} first mentioned the link between (weight zero)
Rota--Baxter maps and dendriform algebras.}}
\end{exam}

Besides the three products $\prec,\succ,\ast$, dendriform algebras
carry naturally other operations. The most interesting, for our
purposes, are the bilinear operations $\rhd$ and $\lhd$ defined
by:
\begin{equation}
\label{def:prelie}
    a \rhd b:= a\succ b-b\prec a,
    \hskip 12mm
    a \lhd b:= a\prec b-b\succ a
\end{equation}
that are left pre-Lie and right pre-Lie, respectively, which means that
we have:
\begin{eqnarray}
    (a\rhd b)\rhd c-a\rhd(b\rhd c)&=& (b\rhd a)\rhd c-b\rhd(a\rhd c),\label{prelie}\\
    (a\lhd b)\lhd c-a\lhd(b\lhd c)&=& (a\lhd c)\lhd b-a\lhd(c\lhd b).
\end{eqnarray}
The associative operation $*$ and the pre-Lie operations $\rhd$,
$\lhd$ all define the same Lie bracket:
\begin{equation}
    [a,b]:=a*b-b*a=a\rhd b-b\rhd a=a\lhd b-b\lhd a.
\end{equation}
\begin{diagram}
    {\rm{dendriform\ alg.}}  & \rTo^{\phantom{m} \lhd,\ \rhd \phantom{m}} & {\rm{pre-Lie\ alg.}}\\
                  \dTo^{*}   &                                            & \dTo_{ [-,-]}\\
    {\rm{associative\ alg.}} & \rTo^{\phantom{m} [-,-] \phantom{m}}       & {\rm{Lie\ alg.}}\\
\end{diagram}

\vspace{0.4cm}

We recursively define on $(A,\prec,\succ)$, augmented by a unit
$1$:
$$
    a \prec 1 := a =: 1 \succ a
    \hskip 12mm
    1 \prec a := 0 =: a \succ 1,
$$
implying $a*1=1*a=a$, the following set of elements for a fixed $x
\in A$:
 \allowdisplaybreaks{
\begin{eqnarray*}
    w^{(0)}_{\prec}(x)&=&w^{(0)}_{\succ}(x)=1,\\
    w^{(n)}_{\prec}(x) &:=& x \prec \bigl(w^{(n-1)}_\prec(x) \bigr),\\
    w^{(n)}_{\succ}(x) &:=& \bigl(w^{(n-1)}_\succ(x)\bigr)\succ x.
\end{eqnarray*}}
Let us recall from Chapoton and
Ronco~\cite{chapoton2000,ronco2000,ronco2002} that, in the free
dendriform dialgebra on one generator $a$, augmented by a
unit element, there is a Hopf algebra structure with respect to
the associative product $*$. The elements
$w^{(n)}_{\succ}:=w^{(n)}_{\succ}(a)$ generate a cocommutative
graded connected Hopf subalgebra $(H,*)$ with
coproduct:
$$
    \Delta(w^{(n)}_{\succ})=w^{(n)}_{\succ} \otimes 1 + 1 \otimes w^{(n)}_{\succ}
    + \sum_{0<m<n} w^{(m)}_{\succ} \otimes w^{(n-m)}_{\succ},
$$
and antipode $S(w^{(n)}_{\succ})=(-1)^{n} w^{(n)}_{\prec}$. It is
actually an easy exercise to check that the $w^{(n)}_{\succ}$
generate a free associative subalgebra of the free dendriform
algebra on $a$ for the $*$ product, so that one can use the
previous formula for the coproduct action on $w^{(n)}_{\succ}$ as
a definition of the Hopf algebra structure on $H$. As an important
consequence, it follows that $H$ is isomorphic, as a Hopf algebra,
to the Hopf algebra of noncommutative symmetric functions
\cite{gelfand1995}

We also define the following set of iterated left and right
pre-Lie products~(\ref{def:prelie}). For $n>0$, let $a_1,\ldots
,a_n \in A$:
 \allowdisplaybreaks{
\begin{eqnarray}
    \ell^{(n)}(a_1,\dots,a_n) &:=&
    \Bigl( \cdots \bigl( (a_1 \rhd a_2) \rhd a_3 \bigr)
    \cdots \rhd a_{n-1} \Bigr) \rhd a_n
\label{leftRBpreLie}\\
    r^{(n)}(a_1,\dots,a_n) &:=&
    a_1 \lhd \Bigl( a_2 \lhd \bigl( a_3 \lhd
    \cdots (a_{n-1} \lhd a_n) \bigr) \cdots \Bigr).
\label{rightRBpreLie}
\end{eqnarray}}
For a fixed single element $a \in A$ we can write more compactly
for $n>0$:
 \allowdisplaybreaks{
\begin{eqnarray}
    \ell^{(n+1)}(a) = \bigl(\ell^{(n)}(a)\bigr)\rhd a
\quad\ {\rm{ and }} \quad\
    r^{(n+1)}(a) = a \lhd \bigl(r^{(n)}(a)\bigr)
\label{def:pre-LieWords}
\end{eqnarray}}
and $\ell^{(1)}(a):=a=:r^{(1)}(a)$.


\section{Dendriform power sums expansions}
\label{dpse}

In the following we would like
to address the theory of solutions of the following two equations
for a fixed $a \in A$:
\begin{equation}
\label{eq:prelie}
     X = 1 + ta \prec X,
    \hskip 12mm
     Y = 1 + Y \succ ta.
\end{equation}
in $A[[t]]$. Formal solutions to these equations are given by the
series of ``left and right non associative power sums'':
$$
    X = \sum_{n \geq 0} t^nw^{(n)}_{\prec}(a)
    \hskip 15mm {\rm{resp.}} \hskip 15mm
    Y = \sum_{n \geq 0} t^nw^{(n)}_{\succ}(a).
$$
Notice that, due to the definition of the Hopf algebra structure
on $ H$, these two series behave as group-like elements with
respect to coproduct $\Delta$ (up to the
extension of the scalars from $k$ to $k[t]$ and the natural
extension of the Hopf algebra structure on $H=\bigoplus_{n\ge
0}H_n$ to its completion $\hat H=\prod_{n\ge 0}H_n$ with respect
to the grading).

Recall now that the Dynkin operator is the linear endomorphism of the
tensor algebra $T(X)$ over an alphabet $X=\{x_1,\ldots
,x_n,\ldots\}$ into itself the action of which on words $y_1
\ldots y_n,\ y_i\in X$ is given by the left-to-right iteration of
the associated Lie bracket:
$$
    D(y_1,\dots,y_n) = [\cdots[[y_1, y_2], y_3]\cdots\!, y_n],
$$
where $[x,y]:=xy-yx$ \cite{reutenauer93}. The Dynkin operator is a
quasi-idempotent: its action on a homogeneous element of
degree~$n$ satisfies $D^2=nD$. The associated projector $D/n$
sends $T_n(X)$, the component of degree~$n$ of the tensor algebra,
to the component of degree~$n$ of the free Lie algebra over $X$.
The tensor algebra is a graded connected cocommutative Hopf
algebra, and it is natural to extend the definition of $D$ to any
such Hopf algebra as the convolution product of the antipode $S$
with the grading operator $N$: $D:=S \star N$
\cite{patreu2002,eunomia2006,EGP07,EMP07}. This applies in
particular in the dendriform context to the Hopf algebra $H$
introduced above. We will write $D_n$ for $D\circ p_n$, where
$p_n$ is the canonical projection from $T(X)$ (resp. $H$) to
$T_n(X)$ (resp. $H_n$).

\begin{lem}\label{lem:dynkin-prelie}
For any integer $n\ge 1$ and for any $a \in A$ we have:
\begin{equation}\label{eq:dynkin-prelie}
    D(w_{\succ}^{(n)}(a)) = {\ell}^{(n)}(a).
\end{equation}
\end{lem}

\begin{proof}
For $n=1$ we have $D(w_{\succ}^{(1)}(a))=D(a)=a={\ell}^{(1)}(a)$.
We then proceed by induction on $n$ and compute:
 \allowdisplaybreaks{
\begin{eqnarray*}
    D(w_{\succ}^{(n)})&=&(S\star N)(w_{\succ}^{(n)})\\
    &=&\sum_{p=0}^{n-1}S(w_{\succ}^{(p)})*N(w_{\succ}^{(n-p)})\\
    &=&\sum_{p=0}^{n-1}S(w_{\succ}^{(p)})*\Big(N(w_{\succ}^{(n-p-1)}) \succ a\Big)
        +\sum_{p=0}^{n-1}S(w_{\succ}^{(p)})*\Big((w_{\succ}^{(n-p-1)})\succ a\Big)\\
    &=&\sum_{p=0}^{n-1}S(w_{\succ}^{(p)})*\Big(N(w_{\succ}^{(n-p-1)})\succ a\Big)
                       +(S\star \mop{Id})(w_{\succ}^{(n)})-S(w_{\succ}^{(n)})\\
    &=&\sum_{p=0}^{n-1}S(w_{\succ}^{(p)})*\Big(N(w_{\succ}^{(n-p-1)}) \succ a\Big)-S(w_{\succ}^{(n)}).
\end{eqnarray*}}
Applying the identities:
 \allowdisplaybreaks{
\begin{eqnarray}
    x*(y\succ z)&=&(x*y)\succ z+x\prec(y\succ z),\\
    S(w_{\succ}^{(n)})&=&-a\prec S(w_{\succ}^{(n-1)}),
\end{eqnarray}}
wet get then:
 \allowdisplaybreaks{
\begin{eqnarray*}
    D(w_{\succ}^{(n)})
        &=&\sum_{p=0}^{n-1}\Big(S(w_{\succ}^{(p)})*N(w_{\succ}^{(n-p-1)})\Big) \succ a
                         +\sum_{p=1}^{n-1}S(w_{\succ}^{(p)})\prec\Big(N(w_{\succ}^{(n-p-1)}) \succ a\Big)
                         -S(w_{\succ}^{(n)})\\
        &=&\Big((S\star N)(w_{\succ}^{(n-1)})\Big)\succ a+\sum_{p=1}^{n-1}
                            S(w_{\succ}^{(p)})\prec\Big(N(w_{\succ}^{(n-p-1)}) \succ a\Big)-S(w_{\succ}^{(n)})\\
        &=&D(w_{\succ}^{(n-1)})\succ a-\sum_{p=1}^{n-1}\Big(a\prec S(w_{\succ}^{(p-1)})\Big)
                            \prec\Big(N(w_{\succ}^{(n-p-1)}) \succ a\Big)+a\prec S(w_{\succ}^{(n-1)})\\
        &=&D(w_{\succ}^{(n-1)})\succ a-\sum_{p=1}^{n-1}a\prec \Big(S(w_{\succ}^{(p-1)})*\big(N(w_{\succ}^{(n-p-1)})\succ a\big)\Big)
            +a\prec S(w_{\succ}^{(n-1)})\\
        &=&D(w_{\succ}^{(n-1)})\succ a-\sum_{p=1}^{n-1}a\prec \Big(S(w_{\succ}^{(p-1)})*(N-\mop{Id})(w_{\succ}^{(n-p)})\Big)
            +a\prec S(w_{\succ}^{(n-1)})\\
        &=&D(w_{\succ}^{(n-1)})\succ a-\sum_{p=0}^{n-1}a\prec \Big(S(w_{\succ}^{(p)})*(N-\mop{Id})(w_{\succ}^{(n-1-p)})\Big)\\
        &=&D(w_{\succ}^{(n-1)})\succ a-a\prec\Big(\big(S\star(N-\mop{Id})\big)(w_{\succ}^{(n-1)})\Big)\\
        &=&D(w_{\succ}^{(n-1)})\succ a-a\prec D(w_{\succ}^{(n-1)})\\
        &=&D(w_{\succ}^{(n-1)})\rhd a={\ell}^{(n-1)}(a)\rhd a={\ell}^{(n)}(a).
\end{eqnarray*}}
\end{proof}
\goodbreak
\begin{thm}
\label{thm:invDynkinRB} \cite{EGP07}, \cite{EMP07} Let
$H=\bigoplus_{n\ge 0}H_n$ be an arbitrary graded connected
cocommutative Hopf algebra over a field of characteristic zero,
and let again $\hat H=\prod_{n\ge 0}H_n$ be its completion with respect
to the grading. The Dynkin operator $D \equiv S\star N$ induces a
bijection between the group~$G(H)$ of group-like elements of $\hat
H$ and the Lie algebra~${\rm{Prim}}(H)$ of primitive elements
in~$\hat H$. The inverse morphism from~${\rm{Prim}}(H)$ to~$G(H)$
is given by
\begin{equation}
\label{eq:invDynkin}
    h = \sum\limits_{n \geq 0} h_n \longmapsto
    \Gamma(h):=\sum\limits_{n \geq 0} \sum\limits_{i_1+\cdots+i_k
              = n \atop i_1,\ldots,i_k>0}\,\frac{h_{i_1} \cdots h_{i_k}}{i_1(i_1 + i_2) \cdots (i_1 + \cdots + i_k)}.
\end{equation}
\end{thm}

Since the element $X$ (resp. $Y$) above is a group-like element in the
Hopf algebra ${\hat H}[[t]]$, lemma
\ref{lem:dynkin-prelie} and theorem \ref{thm:invDynkinRB} imply
the following two identities:

\begin{thm}
\label{dynkin}
We have:
 \allowdisplaybreaks{
\begin{eqnarray}
    w^{(n)}_{\succ}(a) &=&
    \sum\limits_{i_1+ \cdots + i_k=n \atop  i_1,\ldots,i_k>0}
    \frac{\ell^{(i_1)}(a)
    * \cdots *
    \ell^{(i_k)}(a)}
         {i_1(i_1+i_2)\cdots(i_1+\cdots+i_k)},
         \label{solY1}\\
    w^{(n)}_{\prec}(a) &=&
    \sum\limits_{i_1+ \cdots +i_k=n \atop  i_1,\ldots,i_k>0}
    \frac{r^{(i_k)}(a)
    * \cdots *
    r^{(i_1)}(a)}
         {i_1(i_1+i_2)\cdots(i_1+\cdots+i_k)}.
         \label{solX1}
\end{eqnarray}}
\end{thm}

\begin{proof}
Identity (\ref{solY1}) is indeed obtained immediately. Identity
(\ref{solX1}) can be derived easily as follows: consider the
alternative dendriform structure on $A$ defined by:
\begin{equation}
    a\preceq b:=-b\succ a,\hskip 20mm a\succeq b:=-b\prec a.
\end{equation}
The associated associative algebra structure is then defined by:
\begin{equation}
\label{twostars}
    a \ostar b:=-b*a.
\end{equation}
The two pre-Lie operations $\rhd$, $\lhd$ are the same for both
dendriform structures, and are related one to each other by:
\begin{equation}
\label{twoprelie}
    a\lhd b=-b\rhd a.
\end{equation}
We can then obtain (\ref{solX1}) from (\ref{solY1}) and the
identity:
$$
    w_\prec^{(n)}(a)=-w_\succeq^{(n)}(-a).
$$
\end{proof}

An alternative way to deduce (\ref{solX1}) from (\ref{solY1})
consists in applying the antipode $S$ to both sides of
(\ref{solY1}): all $\ell^{(n)}(a)$'s are primitive, as we can see
>from the fact that $D(Y)$ is primitive and from applying lemma
\ref{lem:dynkin-prelie}. The computation follows then easily by
$S\big(\ell^{(n)}(a)\big)=-\ell^{(n)}(a)=(-1)^n r^{(n)}(a)$.

\begin{exam} {\rm{ Let us consider the MAX dendriform algebra $MAX(X
)$ over a countable ordered alphabet
$X=\{x_1,\ldots,x_n,\ldots\}$, (see example~\ref{ex:MAX}), and let us
set $a:=x_1+\cdots+x_n$. Then, we get immediately:
$$
    w_\succ^{(n)}(a)=x_1 \cdots x_n
$$
whereas the multilinear part $m \ell^{(i)}(a)$ of $\ell^{(i)}(a)$
for $i \leq n$ (the component of $\ell^{(i)}(a)$ obtained by
subtracting from $\ell^{(i)}(a)$ the monomials involving non
trivial powers of the letters in $X$, so that e.g. $\ell^{(2)}(x_1
+ x_2) = x_1x_2 - x_2x_1 - x_1^2-x_2^2$ and $m\ell^{(2)}(x_1 +
x_2) = x_1x_2 - x_2x_1$) is given by
$$
    ml^{(i)}(a)=\sum\limits_{1\leq j_1< \cdots <j_i\leq n} D(x_{j_1}\cdots x_{j_i}).
$$
We will abbreviate $D(x_{j_1}\cdots x_{j_i})$ to $D(J)$, where
$J=\{j_1, \ldots ,j_i\}$, so that:
$$
    ml^{(i)}(a)=\sum\limits_{J\subset [n] \atop |J|=i} D(J).
$$
By theorem~\ref{dynkin}, we obtain (keeping only the multilinear
part of the expansion on the right hand side):
$$
    x_1 \cdots x_n = \sum\limits_{i_1+ \cdots + i_k =n
                                  \atop  i_1,\ldots,i_k>0}
                                  \sum\limits_{J_1\coprod  \cdots \coprod J_k=[n] \atop |J_l|=i_l}
    \frac{D(J_1) \cdot \cdots \cdot D(J_k)} {i_1(i_1+i_2) \cdots (i_1+\cdots+i_k)}.
$$
Readers familiar with the Hopf algebraic approach to free Lie
algebras advocated in \cite{reutenauer93} will recognize that this
identity may be rewritten as an expansion of the identity of
$T(X)$ in terms of the Dynkin operator:
$$
    x_1 \cdots x_n = \sum\limits_{i_1+ \cdots + i_k=n \atop  i_1,\ldots,i_k>0}
    \frac{D_{i_1} \star \cdots \star D_{i_k}}
         {i_1(i_1+i_2)\cdots(i_1+\cdots+i_k)}(x_1\cdots x_n),
$$
where $\star$ stands for the convolution product in the set of
linear endomorphisms of $T(X)$, $End(T(X))$
\cite[p.28]{reutenauer93}.}}
\end{exam}

\begin{exam}{\rm{Let us turn to the Malvenuto--Reutenauer
dendriform algebra. Here, we have: $w_\prec^{(n)}(1)=1\cdots n$,
the identity in the symmetric group $S_n$. One can check that
$r^{(n)}(1)$ is the image under the inversion in the symmetric
group $\sigma \rightarrow I(\sigma):=\sigma^{-1}$ (extended
linearly to the group algebra) of the iterated bracket:
$[1,[2,\ldots[n-1,n]\cdots]]$, with the usual convention:
$[i,j]=ij-ji$. We get:
$$
    1 \ldots n = \sum\limits_{i_1+ \cdots + i_k=n
                              \atop  i_1,\ldots,i_k>0}
                              \frac{I([1,\ldots[i_k-1,i_k]\cdots]) \ast \cdots \ast I([1,\ldots[i_1-1,i_1]\cdots])}
                                   {i_1(i_1+i_2)\cdots(i_1+ \cdots + i_k)}.
$$
}}
\end{exam}


\section{Exponential expansions of dendriform power sums}

The following describes an exponential expression of
$Y=Y(t)=\sum_{n \geq 0} t^nw^{(n)}_{\succ}(a)$. An analogous
result is readily derived for $X=X(t)$. Let us define the
exponential map in terms of the associative product,
$\exp^*(x):=\sum_{n \geq 0} x^{*n}/n!$. In $A[[t]]$ we may write
the grading operator $N$ naturally as $t
\partial_t$.

Starting with the fact that $Y(t)$ is group-like in $H$ we easily
find in $A[[t]]$:
\begin{equation}
    D(Y)=Y^{-1} * (t \dot Y),
\end{equation}
hence $\dot Y=Y*\hat{\Cal L}$, with $\hat{\Cal L}:=\hat{\Cal
L}(t)=\frac{D(Y)}{t}=\sum_{n>0}\ell^{(n)}(a)t^{n-1}$. Using
Magnus' expansion~\cite{Magnus} for the solution of first order
linear differential equations, we immediately have
$Y(t)=\exp^*{\Omega(t)}$, $\Omega(t):=\sum_{n>0} \Omega^{(n)}t^n$,
with:
\begin{equation}
    \dot\Omega(t)=\frac{\mop{ad}\Omega(t)}{1-e^{-\smop{ad}\Omega(t)}}\hat{\Cal L}(t).
\end{equation}
This leads to the following well known recursion for $\Omega$:
\begin{equation}
    \Omega(t) = \int_0^t \biggl(\hat{\mathcal{L}}(s) +
    \sum_{n>0} (-1)^n\frac{B_n}{n!} \Big[\hbox{\rm ad} \bigl(\Omega(s)\bigr)\Big]^{n}
    (\hat{\mathcal{L}}(s))\biggr)ds,
\label{Magnus}
\end{equation}
with~$B_n$ the Bernoulli numbers. For $n=1,2,4$ we find
$B_1=-1/2$, $B_2=1/6$ and $B_4=-1/30$, and $b_3=b_5=\cdots=0$. For
the first three terms in the Magnus expansion we find:
\begin{equation}
    \Omega^{(1)}=\ell^{(1)}(a), \;\
    \Omega^{(2)}=\frac{1}{2}\ell^{(2)}(a), \;\
    \Omega^{(3)}=\frac{1}{3}\ell^{(3)}(a) + \frac{1}{12}[\ell^{(1)}(a),\ell^{(2)}(a)],
    \ldots
\label{fullhouse}
\end{equation}
We remark here that the Lie brackets can be written
either in terms of the associative product $*$ or the left
respectively right pre-Lie product.


\section{Lyndon words and dendriform power sums}

Now let $a_1,\ldots ,a_n$ be a collection of elements in $A$. For
any permutation $\sigma\in S_n$ we define the element
$T_\sigma(a_1,\ldots ,a_n)$ as follows: define first the subset
$E_\sigma\subset\{1,\ldots ,n\}$ by $k\in E_\sigma$ if and only if
$\sigma_{k+1}>\sigma_j$ for any $j\le k$, where we abbreviate $\sigma (i)$ to $\sigma_i$. We write $E_\sigma$ in
the increasing order:
\begin{equation*}
    1\le k_1<\cdots < k_p\le n-1.
\end{equation*}
Then we set:
\begin{equation}\label{tsigma}
 T_\sigma(a_1,\ldots ,a_n):=
   \ell^{(k_1)}(a_{\sigma_1},\dots,a_{\sigma_{k_1}}) *\cdots *
   \ell^{(n-k_{p})}(a_{\sigma_{k_p+1}},\dots,a_{\sigma_{n}})
\end{equation}
Quite symmetrically we
define the element $U_\sigma(a_1,\ldots ,a_n)$ by considering
first the subset $F_\sigma\subset\{1,\ldots ,n\}$ defined by $l\in
F_\sigma$ if and only if $\sigma_{l}<\sigma_j$ for any $j\ge l+1$.
We write $F_\sigma$ in the increasing order:
\begin{equation*}
    1\le l_1<\cdots < l_q \le n-1.
\end{equation*}
Then we set:
\begin{equation}
\label{usigma}
 U_\sigma(a_1,\ldots ,a_n):=
   r^{(l_1)}(a_{\sigma_1},\dots,a_{\sigma_{l_1}}) *\cdots *
   r^{(n-l_{q})}(a_{\sigma_{l_q+1}},\dots,a_{\sigma_{n}})
\end{equation}

Following \cite{Lam} it is convenient to encode graphically the
previous statistics on permutations. We write a permutation by
putting a vertical bar (respectively a double bar) after each
element of $E_\sigma$ or $F_\sigma$ according to the case. For
example for the permutation $\sigma=(3261457)$ inside $S_7$ we
have $E_\sigma=\{2,6\}$ and $F_\sigma=\{4,5,6\}$. Putting the
vertical bars:
\begin{equation}
    \sigma=(32|6145|7),\hskip 20mm \sigma=(3261||4||5||7)
\end{equation}
we see that the corresponding elements in $A$ will then be:
 \allowdisplaybreaks{
\begin{eqnarray}
 T_\sigma(a_1,\ldots , a_7)&=&\ell^{(2)}(a_3,a_2)
                              * \ell^{(4)}(a_6,a_1,a_4,a_5)* \ell^{(1)}(a_7)\\
                           &=& (a_3\rhd a_2)*\Big(\big((a_6\rhd a_1)\rhd a_4\big)\rhd a_5\Big)*a_7,\\
 U_\sigma(a_1,\ldots a_7)&=& r^{(4)}(a_3,a_2,a_6,a_1)* r^{(1)}(a_4)
                                * r^{(1)}(a_5)* r^{(1)}(a_7)\\
                         &=&\Big(a_3\lhd\big(a_2\lhd(a_6\lhd a_1)\big)\Big)*a_4*a_5*a_7.
\end{eqnarray}}

\begin{thm}\label{main}
For any $a_1,\ldots ,a_n$ in the dendriform dialgebra $A$ the
following identities hold:
 \allowdisplaybreaks{
\begin{eqnarray}
 \sum_{\sigma\in S_n}\big(\cdots(a_{\sigma_1}\succ a_{\sigma_2})\succ\cdots\big)\succ a_{\sigma_n}
    &=&\sum_{\sigma\in S_n}T_\sigma(a_1,\ldots ,a_n),\label{eq:main}\\
       \sum_{\sigma\in S_n}a_{\sigma_1}\prec\big(\cdots (a_{\sigma_{n-1}}\prec a_{\sigma_n})\cdots\big)
    &=&\sum_{\sigma\in S_n}U_\sigma(a_1,\ldots ,a_n).\label{eq:bismain}
\end{eqnarray}}
\end{thm}
We postpone the proof to the next section, and first give some applications of the identities.

\begin{exam} {\rm{Let us recall first the notion of Lyndon words.
For a given ordered alphabet $X=\{x_1,\ldots,x_n,\ldots\}$, a
Lyndon word is a word (an element $y_1\cdots y_n$, $y_i \in X$, of
the free monoid $X^*$ over $X$) that is strictly less in the
lexicographical ordering than any of its proper right factors
(i.e. strictly less than the $y_i\cdots y_n$, $i>1$). The length
$lgt(w)$ of a word $w$ is the number of letters (with repetitions)
in $w$, so that e.g. $lgt(x_2x_1x_2x_6)=4$.

A fundamental theorem \cite{lot83} asserts that each word $w$ in
$X^*$ has a unique Lyndon factorisation, i.e. can be written
uniquely as
$$
    w=l_1 \cdots l_k
$$
where each $l_i$ is a Lyndon word with $l_1 \geq \cdots \geq l_k$.
The sequence $(lgt(l_1),\ldots,lgt(l_k))$ will be called the
Lyndon sequence of $w$, and written $L(w)$. In the particular case
where $X=[n]^{op}:=\{n,n-1,\ldots,1\}$, the set of positive
integers with the decreasing ordering, it is easily checked that
the Lyndon factorization of a permutation $\sigma \in S_n$, viewed
as the word $\sigma(1)\cdots \sigma(n)$ over $X$ is nothing but
the decomposition introduced above in the definition of $T_\sigma$
so that, for example, the Lyndon factorization of $(3261457)$ is
$32 \cdot 6145 \cdot 7$.

It is also well known that Lyndon words were first introduced to
parameterize bases of the free Lie algebra \cite{mr89}. This
suggests that theorem~\ref{main} might be connected to properties
of bases of free Lie algebras. This is indeed the case, and
dendriform identities provide still another approach and
contribution to their theory and the one of Lyndon words. This
might seem not so surprising after all, since Sch\"utzenberger's
discovery of the dendriform identities has been motivated by the
construction of such bases. Notice however that the dendriform
structure we use below is the MAX dendriform structure of
example~\ref{ex:MAX} and not the one classically used in the
combinatorics of words --that is, the shuffle one. Notice also,
that using different MAX dendriform structures on $T(X)$ (e.g. by
reversing the order on $X$, and so on), would give rise to other
combinatorial formulas than the ones obtained below.

So, let us consider $T(X)$, $X=\{x_1,\ldots,x_n\}$, $x_i <
x_{i+1}$ as a MAX dendriform algebra. Let $\beta$ be an arbitrary
permutation, and set: $a_1:=x_{\beta(1)},\ldots
,a_n:=x_{\beta(n)}$. We have:
$$
    \sum_{\sigma\in S_n}\big(\cdots(a_{\sigma(1)}\succ a_{\sigma(2)})\succ\cdots\big)\succ a_{\sigma(n)}
    =(\cdots((x_1\succ x_2)\succ x_3)\cdots \succ  x_n) = x_1x_2\cdots x_n,
$$
since all the terms in the sum vanish, but one. On the other hand,
for any words $a,b$ in $X^\ast$ without common letters, $a
\triangleright b = [a,b]$ if $max(a) < max(b)$, and zero else. We
get, for any sequence $(a_{i_1},\ldots ,a_{i_k})$, with the $i_j$
distinct:
$$
    (\cdots(a_{i_1}\triangleright a_{i_2})\cdots \triangleright a_{i_k})=0
$$
excepted if the sequence $\beta(i_1),\ldots ,\beta(i_k)$ is
increasing, and then:
$$
    (\cdots(a_{i_1}\triangleright a_{i_2}) \cdots \triangleright a_{i_k})
    =D(a_{i_1}\cdots a_{i_k})=D(x_{\beta (i_1)}\cdots x_{\beta(i_k)}).
$$

Let us write $Lyn(\beta)$ for the set of permutations $\sigma \in
S_n$ such that, if $l_1(\sigma)\cdots l_{k(\sigma)}(\sigma)$,
$l_i(\sigma)=\sigma(n_i)\cdots \sigma(n_{i+1}-1)$ is the Lyndon
factorization of $\sigma$ (as defined above, that is, with respect
to the {\textsl{decreasing order on}} $[n]$), then:
$$
    \forall i \leq k,\ {\beta\circ\sigma (n_i)}<\cdots<{\beta\circ\sigma (n_{i+1}-1)}.
$$

Notice that, for $\sigma =1\ldots n$, with Lyndon factorization of
maximal length $1\cdot 2 \cdot \ \cdots \ \cdot n$, we get $\sigma \in
Lyn(\beta)$ for any $\beta\in S_n$.

For any sequence $S=(i_1,\ldots,i_p)$ of elements of $[n]$, we write
$D(S):=D(x_{i_1}\cdots x_{i_k})$. We also write $\beta(S)$ for
$(\beta(i_1),\ldots,\beta(i_p))$. We get, for any $\beta\in S_n$:
$$
    x_1x_2\cdots x_n=\sum\limits_{\sigma\in Lyn(\beta )}D\big(\beta(l_1(\sigma))\big)\cdots D\big(\beta(l_{k(\sigma)}(\sigma))\big).
$$

A point that should be noticed immediately is that this
decomposition is {\sl not} the classical decomposition of a word
of $X^\ast$ in the Lyndon basis, as described in~\cite{mr89} (and
neither a direct variant thereof). The reason for this is that, by definition
of the Dynkin operator, the opening brackets
inside the blocks $D\big(\beta(l_j(\sigma))\big)$ are all set to the left,
contrarily to what happens in the standard Lyndon factorizations defined in \cite{mr89}.
Since there is a unique
permutation with Lyndon factorization of maximal length, $1\ldots
n$, we also notice that, for any $\beta\in S_n$, this identity is
a rewriting rule expanding $x_1x_2\ldots x_n$ as $x_{\beta
(1)}\cdots x_{\beta (n)}$ plus a sum with integer coefficients of
products of Lie brackets.

Let us consider a few examples. If $\sigma=1\ldots n$, we have
$Lyn(1\ldots n)=\{1\ldots n\}$, and the identity is trivial:
$x_1x_2\cdots x_n=x_1x_2 \ldots x_n$. If $\sigma =\omega= n\ldots
1$, we get:
\begin{equation}
\label{dynkid}
    x_1\cdots x_n = x_n\cdots x_1
                   +\sum\limits_{S_1\coprod \cdots \coprod S_k=[n] \atop
                                 S_i=\{n_1^i<\cdots <n_{|S_i|}^i\}}
                                \prod\limits_{i=1\ldots k}D(n_1^i \cdots n_{|S_i|}^i)
\end{equation}
where the sum runs over all the set partitions of $[n]$, ordered
so that $max(S_1)>\cdots>max (S_k)$, and where the last product is
naturally ordered (the $i-th$ term of the product is
written to the left of the $(i+1)$-th).

This decomposition has a striking property. For brevity sake, we
refer the reader to \cite[Sect.5.6.2]{reutenauer93} for further
details on the notions and results mentioned below. Recall that,
for a given $n$, the Lie brackets $D(S)$, where $S$ runs over all
the words $S=1S'$, $S'$ a permutation of $\{2,\ldots,n\}$ (e.g.
$n=5$, $S=15234$, $S'=5234$) form a basis (over any field of
characteristic zero) of the multilinear part of the free Lie
algebra on $[n]$ --let us call this basis the Dynkin basis. From
this point of view, equation~(\ref{dynkid}) gives nothing but the
expansion of $x_1\cdots x_n$ in the Poincar\'e--Birkhoff--Witt
(PBW) basis (see \cite[Th.0.2]{reutenauer93}) of the multilinear
part of the free Lie algebra over $X$ associated with the Dynkin
basis.

For example, we get:
$$
    x_1x_2x_3=x_3x_2x_1+[x_2,x_3]x_1+x_2[x_1,x_3]+x_3[x_1,x_2]+[[x_1,x_2],x_3],
$$
 \allowdisplaybreaks{
\begin{eqnarray*}
  x_1x_2x_3x_4 &=&   x_4x_3x_2x_1 + x_4x_3[x_1,x_2] + x_4[x_2,x_3]x_1 +x_4x_2[x_1,x_3] + x_4[[x_1,x_2],x_3] + [x_3,x_4]x_2x_1\\
               & & + [x_3,x_4][x_1,x-2] + x_3[x_2,x_4]x_1 + x_3x_2[x_1,x_4] + x_3[[x_1,x_2],x_4] + [[x_2,x_3],x_4]x_1\\
               & & + [x_2,x_4][x_1,x_3] + x_2[[x_1,x_3],x_4] + [[[x_1,x_2],x_3],x_4].
\end{eqnarray*}}

For a general $\sigma$, the expansion allows to rewrite $\sigma$
as a sum of monomials of elements in the Dynkin basis. These
results seem to be new, and connect the fine structure of free Lie
algebras with our structural results on dendriform objects.}}
\end{exam}

\begin{exam}{\rm{ Let us consider now the Malvenuto--Reutenauer
dendriform algebra ${\bf S}_*$. Setting $a_1= \cdots =a_n:=1$ in
theorem~\ref{main}, we get:
$$
    \sum_{\sigma\in S_n} a_{\sigma_1}\prec\big(\cdots (a_{\sigma_{n-1}}\prec a_{\sigma_n})\cdots\big)
                                   =n!\ 1 \ldots n
$$
On the other hand, we know that
$r^{(n)}(1)=I([1,[\cdots[n-1,n]\cdots ]])$ and get (using the symmetry between the definitions of $T_\sigma$ and $U_\sigma$):
$$
    n!\ 1 \ldots n=\sum\limits_{i_1+\cdots +i_k=n, \atop i_j>0}
             |\{\sigma \in S_n,\ L(\sigma )=(i_1, \ldots ,i_k)\}|\
             I([1,[ \cdots [i_k-1,i_k] \cdots ]])\ast  \cdots \ast I([1,[ \cdots [i_1-1,i_1]\cdots]])
$$
Since the Dynkin-type elements $I([1,[ \cdots [i_k-1,i_k] \cdots
]])$ are algebraically independent in the Malvenuto--Reutenauer
algebra (this follows e.g. from \cite[Sect.5]{gelfand1995} and
>from the existence of an embedding of Solomon's descent algebra in
the Malvenuto--Reutenauer algebra \cite{mr95}), one can identify
the coefficients of the last sum with the corresponding
coefficients of the expansion of $1\ldots n$ in example 2 in
section~\ref{dpse}. We get as a corollary an indirect (but
conceptually interesting) computation of the number of
permutations with a given Lyndon sequence:
$$
    |\{\sigma\in S_n,\ L(\sigma )  = (i_1, \ldots ,i_k)\}| = \frac{n!}{i_1(i_1+i_2) \cdots (i_1+ \cdots +i_k)}.
$$
}}
\end{exam}


\section{Proof of the Identity in Theorem~\ref{main}}

Notice that if the left-hand sides of (\ref{eq:main}) and
(\ref{eq:bismain}) are by definition invariant under the
permutation group $S_n$, it is not obvious at all that the
right-hand sides share the same property. The proof of
(\ref{eq:main}) proceeds by induction on the number $n$ of
arguments, and (\ref{eq:bismain}) will be easily deduced from
(\ref{eq:main}). The case $n=2$ reduces to the identity:
\begin{equation}
    a_1\succ a_2+a_2\succ a_1=a_1*a_2+a_2\rhd a_1,
\end{equation}
which immediately follows from the definitions. It is instructive
to detail the case $n=3$: considering the six permutations in
$S_3$:
\begin{equation*}
    (1|2|3),\hskip 8mm (21|3),\hskip 8mm (1|32),\hskip 8mm (321),\hskip 8mm (2|31),\hskip 8mm (312),
\end{equation*}
we then compute, using axioms (\ref{A2}) and (\ref{A3}):
 \allowdisplaybreaks{
\begin{eqnarray*}
    \lefteqn{a_1*a_2*a_3+(a_2\rhd a_1)*a_3+a_1*(a_3\rhd a_2)+(a_3\rhd a_2)\rhd a_1
    +a_2*(a_3\rhd a_1)+(a_3\rhd a_1)\rhd a_2}\\
    &=&(a_1\succ a_2+a_2\succ a_1)*a_3+a_1\succ(a_3\rhd a_2)+(a_3\rhd a_2)\succ a_1
       +a_2\succ(a_3\rhd a_1)+(a_3\rhd a_1)\succ a_2\\
    &=&(a_1\succ a_2)\succ a_3+(a_2\succ a_1)\succ a_3+(a_1\succ a_2)\prec a_3+(a_2\succ a_1)\prec a_3\\
    &\hskip 8mm +&a_1\succ(a_3\succ a_2)-a_1\succ(a_2\prec a_3)
       +(a_3\succ a_2)\succ a_1-(a_2\prec a_3)\succ a_1\\
    &\hskip 8mm +&a_2\succ(a_3\succ a_1)-a_2\succ(a_1\prec a_3)+(a_3\succ a_1)\succ a_2
       -(a_1\prec a_3)\succ a_2\\
    &=&(a_1\succ a_2)\succ a_3+(a_2\succ a_1)\succ a_3+(a_1\succ a_2)\prec a_3
       +(a_2\succ a_1)\prec a_3\\
    &\hskip 8mm +&(a_1\succ a_3)\succ a_2+(a_1\prec a_3)\succ a_2-a_1\succ(a_2\prec a_3)
       +(a_3\succ a_2)\succ a_1-(a_2\prec a_3)\succ a_1\\
    &\hskip 8mm +&(a_2\succ a_3)\succ a_1+(a_2\prec a_3)\succ a_1-a_2\succ(a_1\prec a_3)
       +(a_3\succ a_1)\succ a_2-(a_1\prec a_3)\succ a_2\\
    &=& (a_1\succ a_2)\succ a_3+(a_2\succ a_1)\succ a_3+(a_1\succ a_3)\succ a_2\\
    &\hskip 8mm +&(a_3\succ a_1)\succ a_2+(a_2\succ a_3)\succ a_1+(a_3\succ a_2)\succ a_1.
\end{eqnarray*}}
To start with the proof of the general case, we consider the following
partition of the group $S_n$:
\begin{equation}
\label{partition}
    S_n = S_n^n\amalg\coprod_{j,k=1}^{n-1}S_n^{j,k},
\end{equation}
where $S_n^n$ is the stabilizer of $n$ in $S_n$ , and where
$S_n^{j,k}$ is the subset of those $\sigma\in S_n$ such that
$\sigma_j=n$ and $\sigma_{j+1}=k$. We will set for
$k\in\{1,\ldots,n-1\}$:
\begin{equation}
    S_n^k:=\coprod_{j=1}^{n-1}S_n^{j,k}.
\end{equation}
This is the subset of permutations in $S_n$ in which the two-terms
subsequence $(n,k)$ appears in some place. We have:
\begin{equation}
    S_n=\coprod_{j=1}^{n}S_n^{k}.
\end{equation}
Each $S_n^k$ is in bijective correspondence with $S_{n-1}$, in an
obvious way for $k=n$, by considering the two-term subsequence
$(n,k)$ as a single letter for $k\not =n$. Precisely, in that
case, in the expansion of $\sigma\in S_n$ as a sequence
$\big(\sigma(1),\ldots ,\sigma(n)\big)$, we replace the pair
$(n,k)$ by $n-1$ and any $j\in\{k+1,\ldots,n-1\}$ by $j-1$, so
that, for example, $(2,1,5,3,4)\in S_5^{3,3}$ is sent to
$(2,1,4,3)$. For each $\sigma\in S_n^{k}$ we denote by
$\widetilde\sigma$ its counterpart in $S_{n-1}$. Notice that for
any $k\not =n$ and for any $j\in\{1,\ldots ,n-1\}$, the
correspondence $\sigma\mapsto \sigmat$ sends $S_n^{j,k}$ onto the
subset of $S_{n-1}$ formed by the permutations $\tau$ such that
$\tau_j=n-1$. The following lemma is almost immediate:

\begin{lem}\label{bijections}
For $\sigma\in S_n^n$ we have:
\begin{equation}
    T_\sigma(a_1,\ldots ,a_n)=T_{\widetilde\sigma}(a_1,\ldots ,a_{n-1})*a_n,
\end{equation}
and for $\sigma\in S_n^k, k<n$ we have:
\begin{equation}
    T_\sigma(a_1,\ldots ,a_n)=T_{\widetilde\sigma}(a_1,\ldots ,\widehat{a_k},\ldots, a_{n-1},a_n\rhd a_k),
\end{equation}
where $a_k$ under the hat has been omitted.
\end{lem}

We rewrite the $n-1$-term sequence $(a_1,\ldots
,\widehat{a_k},\ldots, a_{n-1},a_n\rhd a_k)$ as $(c_1^k,\ldots
,c_{n-1}^k)$. We are now ready to compute, using lemma
\ref{bijections} and the induction hypothesis:
 \allowdisplaybreaks{
\begin{eqnarray*}
    \lefteqn{\sum_{\sigma\in S_n}T_\sigma(a_1,\ldots ,a_n) =
                \sum_{k=1}^n\sum_{\sigma\in S_n^k}T_\sigma(a_1,\ldots ,a_n)}\\
    &=& \sum_{\tau\in S_{n-1}}
         \Big(\big(\cdots(a_{\tau_1}\succ a_{\tau_2})\succ\cdots\big)\succ a_{\tau_{n-1}}\Big)*a_n
       +\sum_{k=1}^{n-1}\sum_{\tau\in S_{n-1}}
         \big(\cdots(c^k_{\tau_1}\succ c^k_{\tau_2})\succ\cdots\big)\succ c^k_{\tau_{n-1}}\\
    &=&\hskip -8mm \sum_{\tau\in S_{n-1}}
                    \Big(\big(\cdots(a_{\tau_1}\succ a_{\tau_2})\succ\cdots\big)\succ a_{\tau_{n-1}}\Big) \succ a_n
                  +\sum_{\tau\in S_{n-1}}
                    \Big(\big(\cdots(a_{\tau_1}\succ a_{\tau_2})\succ\cdots\big)\succ a_{\tau_{n-1}}\Big)\prec a_n\\
    &\hskip 8mm +&\sum_{k=1}^{n-1}\sum_{j=1}^{n-1} \sum_{\tau\in S_{n-1} \atop \tau_j=n-1}
    \Big(\cdots\big(\cdots(c_{\tau_1}^k\succ c_{\tau_2}^k)\succ\cdots(a_n\rhd a_k)\big)
    \succ\cdots\Big) \succ c_{\tau_{n-1}}^k,
\end{eqnarray*}}
where $a_n\rhd a_k=c^k_{\tau_j}=c^k_{n-1}$ lies in position $j$. Using the
definition of the pre-Lie operation $\rhd$ and the axiom (\ref{A3}) we get:
 \allowdisplaybreaks{
\begin{eqnarray*}
   \lefteqn{\sum_{\sigma\in S_n}T_\sigma(a_1,\ldots ,a_n)}\\
    &=& \sum_{\tau\in S_{n-1}}
    \Big(\big(\cdots(a_{\tau_1}\succ a_{\tau_2})\succ\cdots\big)\succ a_{\tau_{n-1}}\Big)\succ a_n
        +\sum_{\tau\in S_{n-1}}
    \Big(\big(\cdots(a_{\tau_1}\succ a_{\tau_2})\succ\cdots\big)\succ a_{\tau_{n-1}}\Big)\prec a_n\\
    &\; +\hskip -0.7cm&\sum_{k=1}^{n-1} \sum_{\tau\in S_{n-1} \atop \tau_1=n-1}
    \Big(\cdots\big((a_n\succ a_k)\succ c_{\tau_2}^k\big)\succ\cdots\Big)\succ c_{\tau_{n-1}}^k
    -\sum_{k=1}^{n-1} \sum_{\tau\in S_{n-1}\atop \tau_1=n-1}
    \Big(\cdots\big((a_n\prec a_k)\succ c_{\tau_2}^k\big)\succ\cdots\Big)\succ c_{\tau_{n-1}}^k\\
    &\; +\hskip -0.7cm&\sum_{k=1}^n\sum_{j=2}^{n-1}\sum_{\tau\in S_{n-1} \atop \tau_j=n-1}
    \Bigg(\cdots\bigg(\Big(\big(\cdots(c_{\tau_1}^k\succ c_{\tau_2}^k)\succ\cdots \big)\succ a_n\Big)
    \succ a_k\bigg)\succ\cdots\Bigg)\succ c_{\tau_{n-1}}^k\\
    &\; +\hskip -0.7cm&\sum_{k=1}^n\sum_{j=2}^{n-1}\sum_{\tau\in S_{n-1} \atop \tau_j=n-1}
    \Bigg(\cdots\bigg(\Big(\big(\cdots(c_{\tau_1}^k\succ c_{\tau_2}^k)\succ\cdots \big)\prec
    a_n\Big)\succ a_k\bigg)\succ\cdots\Bigg)\succ c_{\tau_{n-1}}^k\\
    &\; -\hskip -0.7cm&\sum_{k=1}^{n-1}\sum_{j=2}^{n-1} \sum_{\tau\in S_{n-1} \atop \tau_j=n-1}
    \Big(\cdots\big(\cdots(c_{\tau_1}^k\succ c_{\tau_2}^k)\succ\cdots(a_k\prec a_n)\big)\succ
    \cdots\Big)\succ c_{\tau_{n-1}}^k,
\end{eqnarray*}}
where $a_n$ lies in position $j$ (resp. $j+1$) in lines 4 and 5
(resp. in the last line) in the above computation, and where $a_k$
lies in position $j+1$ (resp. $j$) in lines 4 and 5 (resp. in the
last line). We can rewrite this going back to the permutation
group $S_n$ and using the partition (\ref{partition}):
 \allowdisplaybreaks{
\begin{eqnarray*}
    \sum_{\sigma\in S_n}T_\sigma(a_1,\ldots ,a_n)
    &=& \sum_{\sigma\in S_n^n}
        \Big(\big(\cdots(a_{\sigma_1}\succ a_{\sigma_2})\succ\cdots\big)\succ a_{\sigma_{n-1}}\Big)\succ a_{\sigma_n}\\
    &\hskip 8mm +&\sum_{\sigma\in S_n^n}
        \Big(\big(\cdots(a_{\sigma_1}\succ a_{\sigma_2})\succ\cdots\big) \succ a_{\sigma_{n-1}}\big)\prec a_{\sigma_n}\\
    &\hskip 8mm +&\sum_{k=1}^{n-1} \sum_{\sigma\in S_n^{1,k}}
        \Big(\cdots\big((a_{\sigma_1}\succ a_{\sigma_2})\succ a_{\sigma_3}\big)\succ\cdots\Big)\succ a_{\sigma_{n}}\\
    &\hskip 8mm -&\sum_{k=1}^{n-1} \sum_{\sigma\in S_n^{1,k}}
        \Big(\cdots\big((a_{\sigma_1}\prec a_{\sigma_2})\succ a_{\sigma_3}\big)\succ\cdots\Big)\succ a_{\sigma_{n}}\\
    &\hskip 8mm +&\sum_{k=1}^n\sum_{j=2}^{n-1}\sum_{\sigma\in S_n^{j,k}}
        \Bigg(\cdots\bigg(\Big(\big(\cdots(a_{\sigma_1}\succ a_{\sigma_2})\succ\cdots \big)\succ
                                    a_{\sigma_j}\Big)\succ a_{\sigma_{j+1}}\bigg)\succ\cdots\Bigg)\succ a_{\sigma_n}\\
    &\hskip 8mm +&\sum_{k=1}^n\sum_{j=2}^{n-1}\sum_{\sigma\in S_n^{j,k}}
        \Bigg(\cdots\bigg(\Big(\big(\cdots(a_{\sigma_1}\succ a_{\sigma_2})\succ\cdots \big)\prec
                                    a_{\sigma_j}\Big)\succ a_{\sigma_{j+1}}\bigg)\succ\cdots\Bigg)\succ a_{\sigma_n}\\
    &\hskip 8mm -&\sum_{k=1}^{n-1}\sum_{j=2}^{n-1} \sum_{\sigma\in S_n^{j,k}}
        \Big(\cdots\big(\cdots(a_{\sigma_1}\succ a_{\sigma_2})\succ\cdots(a_{\sigma_{j+1}}\prec
                                                                a_{\sigma_j})\big)\succ\cdots\Big)\succ a_{\sigma_n}.
\end{eqnarray*}}
Lines 1,3 and 5 together give the left-hand side of (\ref{eq:main}) whereas lines
2, 4, 6 and 7 cancel. More precisely line 2 cancels with the partial sum
corresponding to $j=n-1$ in line 7, line 4 cancels with the partial sum
corresponding to $j=2$ in line 6, and (for $n\ge 4$), the partial sum
corresponding to some fixed $j\in\{3,\ldots ,n-1\}$ in line 6 cancels with the
partial sum corresponding to $j-1$ in line 7. This proves equality
(\ref{eq:main}).

\smallskip

We could prove {\sl mutatis mutandis\/} (\ref{eq:bismain}) exactly
along the same lines, but we can show that the two versions are in
fact equivalent: The term $T_\sigma^{\succeq}(a_1,\ldots ,a_n)$ is
defined the same way as $T_\sigma(a_1,\ldots ,a_n)$ has been
defined before, but with the dendriform operation $\succeq$
instead of $\succ$.
\begin{lem}
\label{lem:TU} For any $\sigma\in S_n$ and for any $a_1,\ldots ,
a_n\in A$ we have:
\begin{equation}
\label{TU}
    U_{\sigma}(a_1,\ldots,a_n)=(-1)^{n-1}T^{\succeq}_{\omega\sigma\omega}(a_n,\ldots,a_1).
\end{equation}
\end{lem}

\begin{proof}
We denote by $\omega$ the permutation $(n \cdots 21)$ in $S_n$,
and we set $b_j=a_{\omega_j}$, hence:
\begin{equation*}
    (b_1,\ldots ,b_n):=(a_n,\ldots, a_1).
\end{equation*}
 Using
 (\ref{twostars}), (\ref{twoprelie}) and the symmetry:
\begin{equation}
    E_{\omega\sigma\omega}=n-F_\sigma \hbox{ for any }\sigma\in S_n,
\end{equation}
we compute:
 \allowdisplaybreaks{
\begin{eqnarray*}
    \lefteqn{U_{\sigma}(a_1,\ldots,a_n)
    =\Big(a_{\sigma_1}\lhd\big(\cdots (a_{\sigma_{l_1-1}}\lhd a_{\sigma_{l_1}})\big)\cdots\Big)*\cdots
    *\Big(a_{\sigma_{k_q+1}}\lhd\big(\cdots (a_{\sigma_{n-1}}\lhd a_{\sigma_n})\big)\cdots\Big)}\\
    &=&(-1)^{n-1}\Big(\cdots\big((a_{\sigma_n}\rhd a_{\sigma_{n-1}})\rhd\cdots\big)\rhd
            a_{\sigma_{n-k_1-1}}\Big)\ostar \cdots \ostar\Big(\cdots\big((a_{\sigma_{n-k_p}}\rhd
            a_{\sigma_{n-k_p-1}})\rhd \cdots\big)\rhd a_{\sigma_{1}}\Big)\\
    &=&(-1)^{n-1}\Big(\cdots\big((a_{(\sigma\omega)_1}\rhd a_{(\sigma\omega)_2})\rhd\cdots\big)\rhd
            a_{(\sigma\omega)_{k_1}}\Big)\ostar\cdots \ostar\Big(\cdots\big((a_{(\sigma\omega)_{k_p+1}}\rhd
            a_{(\sigma\omega)_{k_p+2}})\rhd \cdots\big)\rhd a_{(\sigma\omega)_{n}}\Big)\\
    &=&(-1)^{n-1}\Big(\cdots\big((b_{(\omega\sigma\omega)_1}\rhd b_{(\omega\sigma\omega)_2})\rhd\cdots\big)\rhd
            b_{(\omega\sigma\omega)_{k_1}}\Big)\ostar\cdots \ostar\Big(\cdots\big((b_{(\omega\sigma\omega)_{k_p+1}}\rhd
            b_{(\omega\sigma\omega)_{k_p+2}})\rhd \cdots\big)\rhd b_{(\omega\sigma\omega)_{n}}\Big)\\
    &=&(-1)^{n-1}T^{\succeq}_{\omega\sigma\omega}(b_1,\ldots ,b_n).
\end{eqnarray*}}
\end{proof}
Hence we compute, using successively the $S_n$-invariance,
equation (\ref{eq:main}) and lemma \ref{lem:TU}:
 \allowdisplaybreaks{
\begin{eqnarray*}
 \sum_{S_n}a_{\sigma_1}\prec\big(\cdots\prec(a_{\sigma_{n-1}}\prec a_{\sigma_n})\cdots\big)
 &=&\sum_{S_n}b_{(\omega\sigma)_1}\prec\big(\cdots\prec(b_{(\omega\sigma)_{n-1}}\prec b_{(\omega\sigma)_n})\cdots\big)\\
 &=&(-1)^{n-1}\sum_{S_n}\big(\cdots(b_{(\omega\sigma)_n}\succeq b_{(\omega\sigma)_{n-1}})\succeq\cdots\big)\succeq b_{(\omega\sigma)_1}\\
 &=&(-1)^{n-1}\sum_{S_n}T^{\succeq}_{\omega\sigma\omega}(b_1,\ldots ,b_n)\\
 &=&\sum_{S_n}U_\sigma(a_1,\ldots,a_n),
\end{eqnarray*}}
which finishes the proof of the theorem.


\section{Rota--Baxter algebras and dendriform algebras}
\label{RB-Dendriform-Lam}

Recall \cite{E} that an associative Rota--Baxter
algebra (over a field $k$) is an associative algebra $(A,.)$
endowed with a linear map $R:A\to A$ subject to the following
relation:
\begin{equation}\label{RB}
R(a)R(b)=R\big(R(a)b+aR(b)+\theta ab\big).
\end{equation}
where $\theta\in k$. The map $R$ is called a {\sl Rota--Baxter
operator of weight $\theta$\/}. The map $\tilde{R}:=-\theta id -R$
also is a weight $\theta$ Rota--Baxter map.

\begin{prop}\cite {E}
Any Rota--Baxter algebra gives rise to two dendriform dialgebra
structures given by:
 \allowdisplaybreaks{
\begin{eqnarray}
    a\prec b&:=&aR(b)+\theta ab=-a\tilde{R}(b),\hskip 8mm a\succ b:=R(a)b,\\
    a\prec' b&:=&aR(b),\hskip 33mm a\succ' b:=R(a)b+\theta ab=-\tilde{R}(a)b.
\end{eqnarray}}
\end{prop}

The associated associative product $*$ is given for both
structures by $a*b=aR(b)+R(a)b+\theta ab$. It is sometimes called
the ``double Rota--Baxter product'', and verifies:
\begin{equation}
    R(a*b)=R(a)R(b),
\end{equation}
which is just a reformulation of the Rota--Baxter relation
(\ref{RB}).
\begin{rmk}\cite {E}
In fact, by splitting again the binary operation $\prec$ (or
alternatively $\succ'$), any Rota--Baxter algebra is {\rm
tri-dendriform\/}, in the sense that the Rota--Baxter structure
yields three binary operations $<,\diamond,
>$ subject to axioms refining the axioms of dendriform dialgebras~\cite{LodayRonco04}.
The three binary operations are defined by $a<b=aR(b)$, $a\diamond
b=\theta ab$ and $a>b=R(a)b$. Choosing to put the operation
$\diamond$ to the $<$ or $>$ side gives rise to the two dendriform
structures above.
\end{rmk}

Theorem \ref{main} in the Rota--Baxter setting thus takes the
following form:

\begin{cor}\label{cor:RB-Dend-Lam}
Let $(A,R)$ be a weight $\theta$ Rota--Baxter algebra, let $*$ be
the double Rota--Baxter product defined above. Then, with the
notations of section \ref{formulation} we have:
 \allowdisplaybreaks{
\begin{eqnarray}
    \sum_{\sigma\in S_n}
    R\Big(\cdots R\big(R(a_{\sigma_1})a_{\sigma_2}\big)\cdots a_{\sigma_{n-1}}\Big)a_{\sigma_n}
    &=&\sum_{\sigma\in S_n}
        T_\sigma(a_1,\ldots ,a_n),
        \label{eq:mainRB1}\\
    \sum_{\sigma\in S_n}a_{\sigma_1}
    R\Big(a_{\sigma_2}\cdots R\big(a_{\sigma_{n-1}}R(a_{\sigma_n})\big)\cdots\Big)
    &=&\sum_{\sigma\in S_n}U'_\sigma(a_1,\ldots,a_n),
\label{eq:mainRB2}
\end{eqnarray}}
where $U'_\sigma(a_1,\ldots,a_n)$ is defined the same way as
$U_\sigma(a_1,\ldots,a_n)$ previously, but with the dendriform
structure $(A,\prec',\succ')$. The pre-Lie operation $\rhd$ (resp.
$\lhd'$) involved in the right-hand side of equality
(\ref{eq:mainRB1})(resp. (\ref{eq:mainRB2})) is given by:
\begin{equation}
\label{RBpreLie}
    a\rhd b=R(a)b-bR(a)-\theta ba = [R(a),b]-\theta ba,
    \quad\hbox{ {\rm{resp.}} }\quad
    a \lhd' b=aR(b)-R(b)a-\theta ba = [a,R(b)]-\theta ba.
\end{equation}
\end{cor}
Applying the Rota--Baxter operator $R$ to both sides of these two
identities gives back the noncommutative Bohnenblust--Spitzer
identity which has been announced in \cite{EGP07} and proved in
\cite{EMP07} (Theorem 7.1). What we have obtained in theorem
\ref{main} is thus an extension of this noncommutative
Bohnenblust--Spitzer identity to the dendriform setting.\\

In the weight $\theta=0$ case, the pre-Lie operation reduces to
$a\rhd b=[R(a),b]=- b \lhd' a=-b\lhd a$. This case, in the form
(\ref{eq:mainRB2}), has been handled by C.S.~Lam in~\cite{Lam}, in
the concrete situation when $A$ is a function space on the real
line, and when $R(f)$ is the primitive of $f$ which vanishes at a
fixed $T\in{\mathbb R}$. The formulation of theorem~\ref{main} in
the general dendriform setting thus permits an application to
Rota--Baxter operators of any weight $\theta$.

In the particular case of a commutative Rota--Baxter algebra the
identities in corollary~\ref{cor:RB-Dend-Lam} reduce to one since
both Rota--Baxter pre-Lie products~(\ref{RBpreLie}) agree. One
recovers the classical Spitzer identity of fluctuation theory, and
Rota's generalization thereof to arbitrary commutative
Rota--Baxter algebras \cite{Rota,RS72}.

\vspace{0.5cm}

\textbf{Acknowledgements}

\smallskip

The first named author acknowledges greatly the support by the
European Post-Doctoral Institute. The present work received
support from the ANR grant AHBE 05-42234.

\end{document}